\documentclass{amsart}
\usepackage[latin1]{inputenc}
\usepackage{amsmath,amssymb,graphicx,setspace,verbatim}
\usepackage{color}
\usepackage{comment}
\usepackage{soul}
\usepackage{appendix}
\theoremstyle{comment}
\usepackage[color,matrix,arrow]{xy}
\usepackage{cleveref}
\newtheorem*{mcomment}{\color{cyan}{Comment}}

\usepackage{multicol}
\usepackage{multirow}
\input xy
\xyoption{all}

\newtheorem{theorem}{Theorem}[section]
\newtheorem{problem}{Problem}[section]

\newtheorem{lemma}[theorem]{Lemma}
\newtheorem{corollary}[theorem]{Corollary}
\newtheorem{proposition}[theorem]{Proposition}
\theoremstyle{definition}

\usepackage[pagewise]{lineno}
\begin{document}

\title{The degrees of toroidal regular proper hypermaps}

\author{Maria Elisa Fernandes}
\address{Maria Elisa Fernandes, Department of Mathematics,
University of Aveiro,
Aveiro,
Portugal}
\email{maria.elisa@ua.pt}

\author{Claudio Alexandre Piedade}
\address{Claudio Alexandre Piedade, Department of Mathematics,
University of Aveiro,
Aveiro,
Portugal}
\email{claudio.a.piedade@ua.pt}

\begin{abstract}
Recently the classification of all possible faithful transitive permutation representations of the group of symmetries of a regular toroidal map was accomplished.
In this paper we complete this investigation on a surface of genus 1 considering the group of a regular toroidal hypermap of type $(3,3,3)$ that is a subgroup of index $2$ of the group of symmetries of a toroidal map of type $\{6,3\}$.
\end{abstract}

\maketitle

\noindent \textbf{Keywords:} Regular Polytopes, Regular Toroidal Maps, Regular Toroidal Hypermaps, Permutation Groups.

\noindent \textbf{2010 Math Subj. Class:} 52B11, 05E18,  20B25.

\section{Introduction}\label{back} 

By Cayley's theorem, every group is isomorphic to some permutation group.
A finite group $G$ has a \emph{faithful permutation representation of degree $n$} if there exists a monomorphism from $G$ into the symmetric group $S_n$, or equivalently,  if $G$ acts faithfully on a set of $n$ points.
In this paper, only transitive actions are considered and faithful transitive permutation representations of a group $G$ correspond to core-free subgroups of $G$, that is, groups containing nontrivial normal subgroups. The stabilizer of a faithful transitive permutation representation is core-free and conversely, the action on the cosets of a core-free subgroup is faithful and transitive. 

The minimal degree of a faithful permutation representation of $G$ has been a subject of extensive study. In \cite{1971J} it was shown that a faithful permutation representation of a simple group with minimal degree is primitive. The minimal degree of a faithful (transitive) permutation representation is known for all simple groups \cite [Theorem 5.2.2]{1990KL}. 

We have particular interest on the study of the transitive permutation representations of the groups of abstract regular polytopes, that are quotients of Coxeter groups \cite{ARP}, or more generally, of the groups of regular hypertopes \cite{FLW1}. The minimal faithful permutation representations of irreducible Coxeter groups, the groups of spherical polytopes, was recently determined in \cite{S14}.

This paper is a sequel to \cite{2019FP} in which faithful transitive permutation representations of the groups of symmetries of toroidal regular maps is determined.
In the present paper we complete the classification of toroidal regular hypermaps, answering a question made by Gareth Jones, in the Bled Conference in Graph Theory 2018, where the results accomplished in  \cite{2019FP} were presented.

The results can be summarized as follows. Consider $s\geq 2$ and $s=lcm(a,b)$:
\begin{itemize}
 \item for the hypermap $(3,3,3)_{(s,0)}$, the possible degrees are $s^2$, $2s^2$, $3ds$ and $6ab$, where $d$ is a divisor of $s$;
 \item for the hypermap $(3,3,3)_{(s,s)}$, the possible degrees are those of the hypermap $(3,3,3)_{(s,0)}$ multiplied by 3.
\end{itemize}
We observe that this result is not obtain directly from the classification of all the possible degrees of the toroidal map $\{6,3\}$, as after factorization some faithful permutation representations are lost.

\section{Toroidal hypermaps}\label{back}

Consider a regular tessellation of the plane by identical hexagons, whose full symmetry group is the Coxeter group $[6,3]$, generated by three reflections $\tau_0$, $\tau_1$ and $\tau_2$, as shown in Figure \ref{reflethexagon}.

\begin{figure}
$$\xymatrix@-1.8pc{
 &&*{}\ar@{-}[drr]\ar@{-}[dll]&&&&*{}\ar@{-}[drr]\ar@{-}[dll]&&&&*{}\ar@{-}[drr]\ar@{-}[dll]&&&&*{}\ar@{-}[drr]\ar@{-}[dll]&&*{}\ar@{.}[dddddddddddddddddddddd]^(.001){\tau_1}&&*{}\ar@{-}[drr]\ar@{-}[dll]&&&&*{}\ar@{-}[drr]\ar@{-}[dll]&*{}\ar@{.}[ddddddddddddddddddddddllllllllllllll]^(.001){\tau_0}&&&*{}\ar@{-}[drr]\ar@{-}[dll]&&&&*{}\ar@{-}[drr]\ar@{-}[dll]&&\\
 *{}\ar@{-}[dd]&&&&*{}\ar@{-}[dd]&&&&*{}\ar@{-}[dd]&&&&*{}\ar@{-}[dd]&&&&*{}\ar@{-}[dd]&&&&*{}\ar@{-}[dd]&&&&*{}\ar@{-}[dd]&&&&*{}\ar@{-}[dd]&&&&*{}\ar@{-}[dd]\\
 \\
 *{}\ar@{-}[drr]&&&&*{}\ar@{-}[drr]\ar@{-}[dll]&&&&*{}\ar@{-}[drr]\ar@{-}[dll]&&&&*{}\ar@{-}[drr]\ar@{-}[dll]&&&&*{}\ar@{-}[drr]\ar@{-}[dll]&&&&*{}\ar@{-}[drr]\ar@{-}[dll]&&&&*{}\ar@{-}[drr]\ar@{-}[dll]&&&&*{}\ar@{-}[drr]\ar@{-}[dll]&&&&*{}\ar@{-}[dll]\\
  &&*{}\ar@{-}[dd]&&&&*{}\ar@{-}[dd]&&&&*{}\ar@{-}[dd]&&&&*{}\ar@{-}[dd]&&&&*{}\ar@{-}[dd]&&&&*{}\ar@{-}[dd]&&&&*{}\ar@{-}[dd]&&&&*{}\ar@{-}[dd]\\
 \ar@{.}[rrrrrrrrrrrrrrrrrrrrrrrrrrrrrrrrdddddddddddddddd]^(.01){\tau_2}\\
 &&*{}\ar@{-}[drr]\ar@{-}[dll]&&&&*{}\ar@{-}[drr]\ar@{-}[dll]&&&&*{}\ar@{-}[drr]\ar@{-}[dll]&&&&*{}\ar@{-}[drr]\ar@{-}[dll]&&&&*{}\ar@{-}[drr]\ar@{-}[dll]&&&&*{\bullet}\ar@{-}[drr]\ar@{-}[dll]&&&&*{}\ar@{-}[drr]\ar@{-}[dll]&&&&*{}\ar@{-}[drr]\ar@{-}[dll]&&\\
 *{}\ar@{-}[dd]&&&&*{}\ar@{-}[dd]&&&&*{}\ar@{-}[dd]&&&&*{}\ar@{-}[dd]&&&&*{}\ar@{-}[dd]&&&&*{}\ar@{-}[dd]&&&&*{}\ar@{-}[dd]&&&&*{}\ar@{-}[dd]&&&&*{}\ar@{-}[dd]\\
 \\
 *{}\ar@{-}[drr]&&&&*{}\ar@{-}[drr]\ar@{-}[dll]&&&&*{\bullet}\ar@{-}[drr]\ar@{-}[dll]\ar@{--}[rrrrrrrrrrrrrruuu]&&&&*{}\ar@{-}[drr]\ar@{-}[dll]&&&&*{}\ar@{-}[drr]\ar@{-}[dll]&&&&*{}\ar@{-}[drr]\ar@{-}[dll]&&&&*{}\ar@{-}[drr]\ar@{-}[dll]&&&&*{}\ar@{-}[drr]\ar@{-}[dll]&&&&*{}\ar@{-}[dll]\\
  &&*{}\ar@{-}[dd]&&&&*{}\ar@{-}[dd]&&&&*{}\ar@{-}[dd]&&&&*{}\ar@{-}[dd]&&&&*{}\ar@{-}[dd]&&&&*{}\ar@{-}[dd]&&&&*{}\ar@{-}[dd]&&&&*{}\ar@{-}[dd]\\
 \\
  &&*{}\ar@{-}[drr]\ar@{-}[dll]&&&&*{}\ar@{-}[drr]\ar@{-}[dll]&&&&*{}\ar@{-}[drr]\ar@{-}[dll]&&&&*{}\ar@{-}[drr]\ar@{-}[dll]&&&&*{}\ar@{-}[drr]\ar@{-}[dll]&&&&*{}\ar@{-}[drr]\ar@{-}[dll]&&&&*{}\ar@{-}[drr]\ar@{-}[dll]&&&&*{}\ar@{-}[drr]\ar@{-}[dll]&&\\
 *{}\ar@{-}[dd]&&&&*{}\ar@{-}[dd]&&&&*{}\ar@{-}[dd]&&&&*{}\ar@{-}[dd]&&&&*{}\ar@{-}[dd]&&&&*{}\ar@{-}[dd]&&&&*{}\ar@{-}[dd]&&&&*{}\ar@{-}[dd]&&&&*{}\ar@{-}[dd]\\
 \\
 *{}\ar@{-}[drr]&&&&*{}\ar@{-}[drr]\ar@{-}[dll]&&&&*{}\ar@{-}[drr]\ar@{-}[dll]&&&&*{}\ar@{-}[drr]\ar@{-}[dll]&&&&*{}\ar@{-}[drr]\ar@{-}[dll]&&&&*{\bullet}\ar@{-}[drr]\ar@{-}[dll]\ar@{--}[rruuuuuuuuu]_(.99){(s-t,s+2t)} &&&&*{}\ar@{-}[drr]\ar@{-}[dll]&&&&*{}\ar@{-}[drr]\ar@{-}[dll]&&&&*{}\ar@{-}[dll]\\
  &&*{}\ar@{-}[dd]&&&&*{}\ar@{-}[dd]&&&&*{}\ar@{-}[dd]&&&&*{}\ar@{-}[dd]&&&&*{}\ar@{-}[dd]&&&&*{}\ar@{-}[dd]&&&&*{}\ar@{-}[dd]&&&&*{}\ar@{-}[dd]\\
 \\
  &&*{}\ar@{-}[drr]\ar@{-}[dll]&&&&*{\bullet}\ar@{-}[drr]\ar@{-}[dll]\ar@{--}[rrrrrrrrrrrrrruuu]_(.01){(0,0)}_(.99){(s,t)}\ar@{--}[rruuuuuuuuu]^(.99){(-t,s+t)}&&&&*{}\ar@{-}[drr]\ar@{-}[dll]&&&&*{}\ar@{-}[drr]\ar@{-}[dll]&&&&*{}\ar@{-}[drr]\ar@{-}[dll]&&&&*{}\ar@{-}[drr]\ar@{-}[dll]&&&&*{}\ar@{-}[drr]\ar@{-}[dll]&&&&*{}\ar@{-}[drr]\ar@{-}[dll]&&\\
 *{}\ar@{-}[dd]&&&&*{}\ar@{-}[dd]&&&&*{}\ar@{-}[dd]&&&&*{}\ar@{-}[dd]&&&&*{}\ar@{-}[dd]&&&&*{}\ar@{-}[dd]&&&&*{}\ar@{-}[dd]&&&&*{}\ar@{-}[dd]&&&&*{}\ar@{-}[dd]\\
 \\
 *{}\ar@{-}[drr]&&&&*{}\ar@{-}[drr]\ar@{-}[dll]&&&&*{}\ar@{-}[drr]\ar@{-}[dll]&&&&*{}\ar@{-}[drr]\ar@{-}[dll]&&&&*{}\ar@{-}[drr]\ar@{-}[dll]&&&&*{}\ar@{-}[drr]\ar@{-}[dll]&&&&*{}\ar@{-}[drr]\ar@{-}[dll]&&&&*{}\ar@{-}[drr]\ar@{-}[dll]&&&&*{}\ar@{-}[dll]\\
  &&*{}&&&&*{}&&&&*{}&&&&*{}&&&&*{}&&&&*{}&&&&*{}&&&&*{}&&&&*{}\\
}$$
\caption{Toroidal map of type $\{6,3\}$}\label{reflethexagon}
\end{figure}
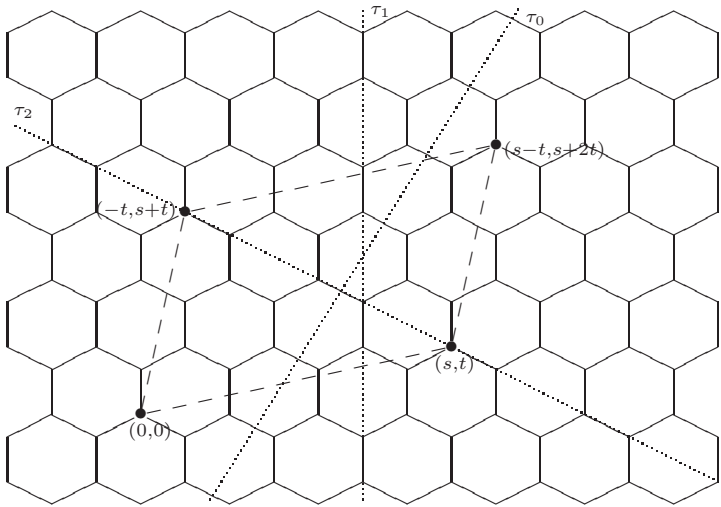

By identifying opposite sides of a parallelogram with vertices $(0,0)$, $(s,t)$, $(-t,s+t)$ and $(s-t,s+2t)$ of the tessellation, we obtain the toroidal map $\{6,3\}_{(s,t)}$, with $F = s^2+st+t^2$ faces, $3F$ edges and $2F$ vertices. This map is said to be regular when the group of symmetries acts regularly on the set of flags of the map (triples of mutually incident vertex, edge and face), having $\tau_0$, $\tau_1$ and $\tau_2$ be reflections of $\{6,3\}_{(s,t)}$, i.e. only if $st(s-t)=0$.
Therefore, two families of toroidal regular maps of type $\{6,3\}$ arise: $\{6,3\}_{(s,0)}$ and $\{6,3\}_{(s,s)}$, which are obtained by factorization of the Coxeter group $[6,3]$ by $(\tau_0\tau_1\tau_2)^{2s}$ and $(\tau_0\tau_1\tau_0\tau_1\tau_2)^{2s}$, respectively. The number of flags of $\{6,3\}_{(s,0)}$ is $12s^2$ while the number of flags of $\{6,3\}_{(s,s)}$ is $36s^2$.

A hypermap can be defined as an embedding of a bipartite graph (or of a hypergraph) into a compact surface. The bipartition of vertices determines two types of vertices, hypervertices and hyperedges. A toroidal hypermap is obtained from a map of type $\{6,3\}$ by considering a bipartition on the set of its vertices (see Figure \ref{hypermap333}). The toroidal hypermap construct from $\{6,3\}_{(s,t)}$ is denoted by $(3,3,3)_{(s,t)}$.  The group $G$ of symmetries of the hypermap $(3,3,3)_{(s,t)}$ is a subgroup of index $2$ of the group of the map $\{6,3\}_{(s,t)}$,
$$G := \langle \rho_0, \rho_1, \rho_2\rangle, \mbox{ where }\rho_0 := \tau_0\tau_1\tau_0,\, \rho_1 := \tau_1\mbox{ and }\rho_2 := \tau_2.$$

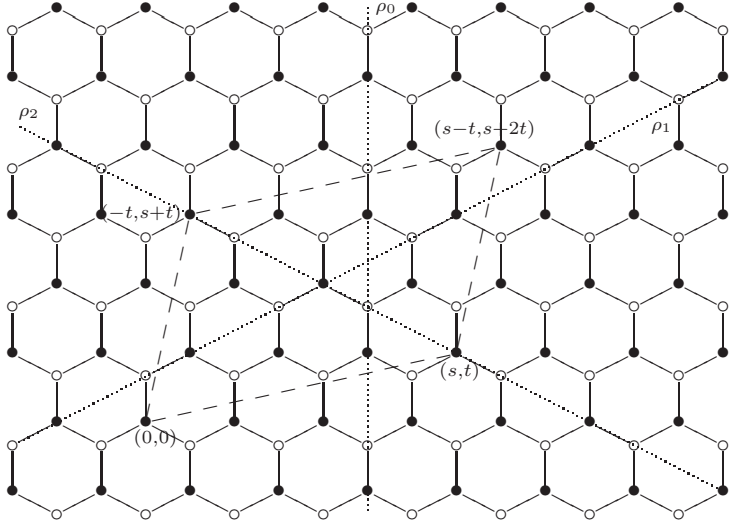
\begin{figure}
 $$\xymatrix@-1.8pc{
 &&*{\bullet}\ar@{-}[drr]\ar@{-}[dll]&&&&*{\bullet}\ar@{-}[drr]\ar@{-}[dll]&&&&*{\bullet}\ar@{-}[drr]\ar@{-}[dll]&&&&*{\bullet}\ar@{-}[drr]\ar@{-}[dll]&&*{}\ar@{.}[dddddddddddddddddddddd]^(.001){\rho_0}&&*{\bullet}\ar@{-}[drr]\ar@{-}[dll]&&&&*{\bullet}\ar@{-}[drr]\ar@{-}[dll]&&&&*{\bullet}\ar@{-}[drr]\ar@{-}[dll]&&&&*{\bullet}\ar@{-}[drr]\ar@{-}[dll]&&\\
 *{\circ}\ar@{-}[dd]&&&&*{\circ}\ar@{-}[dd]&&&&*{\circ}\ar@{-}[dd]&&&&*{\circ}\ar@{-}[dd]&&&&*{\circ}\ar@{-}[dd]&&&&*{\circ}\ar@{-}[dd]&&&&*{\circ}\ar@{-}[dd]&&&&*{\circ}\ar@{-}[dd]&&&&*{\circ}\ar@{-}[dd]\\
 \\
 *{\bullet}\ar@{-}[drr]&&&&*{\bullet}\ar@{-}[drr]\ar@{-}[dll]&&&&*{\bullet}\ar@{-}[drr]\ar@{-}[dll]&&&&*{\bullet}\ar@{-}[drr]\ar@{-}[dll]&&&&*{\bullet}\ar@{-}[drr]\ar@{-}[dll]&&&&*{\bullet}\ar@{-}[drr]\ar@{-}[dll]&&&&*{\bullet}\ar@{-}[drr]\ar@{-}[dll]&&&&*{\bullet}\ar@{-}[drr]\ar@{-}[dll]&&&&*{\bullet}\ar@{-}[dll]\ar@{.}[ddddddddddddddddllllllllllllllllllllllllllllllll]^(.1){\rho_1}\\
  &&*{\circ}\ar@{-}[dd]&&&&*{\circ}\ar@{-}[dd]&&&&*{\circ}\ar@{-}[dd]&&&&*{\circ}\ar@{-}[dd]&&&&*{\circ}\ar@{-}[dd]&&&&*{\circ}\ar@{-}[dd]&&&&*{\circ}\ar@{-}[dd]&&&&*{\circ}\ar@{-}[dd]\\
 \ar@{.}[rrrrrrrrrrrrrrrrrrrrrrrrrrrrrrrrdddddddddddddddd]^(.01){\rho_2}\\
 &&*{\bullet}\ar@{-}[drr]\ar@{-}[dll]&&&&*{\bullet}\ar@{-}[drr]\ar@{-}[dll]&&&&*{\bullet}\ar@{-}[drr]\ar@{-}[dll]&&&&*{\bullet}\ar@{-}[drr]\ar@{-}[dll]&&&&*{\bullet}\ar@{-}[drr]\ar@{-}[dll]&&&&*{\bullet}\ar@{-}[drr]\ar@{-}[dll]&&&&*{\bullet}\ar@{-}[drr]\ar@{-}[dll]&&&&*{\bullet}\ar@{-}[drr]\ar@{-}[dll]&&\\
 *{\circ}\ar@{-}[dd]&&&&*{\circ}\ar@{-}[dd]&&&&*{\circ}\ar@{-}[dd]&&&&*{\circ}\ar@{-}[dd]&&&&*{\circ}\ar@{-}[dd]&&&&*{\circ}\ar@{-}[dd]&&&&*{\circ}\ar@{-}[dd]&&&&*{\circ}\ar@{-}[dd]&&&&*{\circ}\ar@{-}[dd]\\
 \\
 *{\bullet}\ar@{-}[drr]&&&&*{\bullet}\ar@{-}[drr]\ar@{-}[dll]&&&&*{\bullet}\ar@{-}[drr]\ar@{-}[dll]\ar@{--}[rrrrrrrrrrrrrruuu]^(.99){(s-t,s+2t)}&&&&*{\bullet}\ar@{-}[drr]\ar@{-}[dll]&&&&*{\bullet}\ar@{-}[drr]\ar@{-}[dll]&&&&*{\bullet}\ar@{-}[drr]\ar@{-}[dll]&&&&*{\bullet}\ar@{-}[drr]\ar@{-}[dll]&&&&*{\bullet}\ar@{-}[drr]\ar@{-}[dll]&&&&*{\bullet}\ar@{-}[dll]\\
  &&*{\circ}\ar@{-}[dd]&&&&*{\circ}\ar@{-}[dd]&&&&*{\circ}\ar@{-}[dd]&&&&*{\circ}\ar@{-}[dd]&&&&*{\circ}\ar@{-}[dd]&&&&*{\circ}\ar@{-}[dd]&&&&*{\circ}\ar@{-}[dd]&&&&*{\circ}\ar@{-}[dd]\\
 \\
  &&*{\bullet}\ar@{-}[drr]\ar@{-}[dll]&&&&*{\bullet}\ar@{-}[drr]\ar@{-}[dll]&&&&*{\bullet}\ar@{-}[drr]\ar@{-}[dll]&&&&*{\bullet}\ar@{-}[drr]\ar@{-}[dll]&&&&*{\bullet}\ar@{-}[drr]\ar@{-}[dll]&&&&*{\bullet}\ar@{-}[drr]\ar@{-}[dll]&&&&*{\bullet}\ar@{-}[drr]\ar@{-}[dll]&&&&*{\bullet}\ar@{-}[drr]\ar@{-}[dll]&&\\
 *{\circ}\ar@{-}[dd]&&&&*{\circ}\ar@{-}[dd]&&&&*{\circ}\ar@{-}[dd]&&&&*{\circ}\ar@{-}[dd]&&&&*{\circ}\ar@{-}[dd]&&&&*{\circ}\ar@{-}[dd]&&&&*{\circ}\ar@{-}[dd]&&&&*{\circ}\ar@{-}[dd]&&&&*{\circ}\ar@{-}[dd]\\
 \\
 *{\bullet}\ar@{-}[drr]&&&&*{\bullet}\ar@{-}[drr]\ar@{-}[dll]&&&&*{\bullet}\ar@{-}[drr]\ar@{-}[dll]&&&&*{\bullet}\ar@{-}[drr]\ar@{-}[dll]&&&&*{\bullet}\ar@{-}[drr]\ar@{-}[dll]&&&&*{\bullet}\ar@{-}[drr]\ar@{-}[dll]\ar@{--}[rruuuuuuuuu]&&&&*{\bullet}\ar@{-}[drr]\ar@{-}[dll]&&&&*{\bullet}\ar@{-}[drr]\ar@{-}[dll]&&&&*{\bullet}\ar@{-}[dll]\\
  &&*{\circ}\ar@{-}[dd]&&&&*{\circ}\ar@{-}[dd]&&&&*{\circ}\ar@{-}[dd]&&&&*{\circ}\ar@{-}[dd]&&&&*{\circ}\ar@{-}[dd]&&&&*{\circ}\ar@{-}[dd]&&&&*{\circ}\ar@{-}[dd]&&&&*{\circ}\ar@{-}[dd]\\
 \\
  &&*{\bullet}\ar@{-}[drr]\ar@{-}[dll]&&&&*{\bullet}\ar@{-}[drr]\ar@{-}[dll]\ar@{--}[rrrrrrrrrrrrrruuu]_(.01){(0,0)}_(.99){(s,t)}\ar@{--}[rruuuuuuuuu]^(.99){(-t,s+t)}&&&&*{\bullet}\ar@{-}[drr]\ar@{-}[dll]&&&&*{\bullet}\ar@{-}[drr]\ar@{-}[dll]&&&&*{\bullet}\ar@{-}[drr]\ar@{-}[dll]&&&&*{\bullet}\ar@{-}[drr]\ar@{-}[dll]&&&&*{\bullet}\ar@{-}[drr]\ar@{-}[dll]&&&&*{\bullet}\ar@{-}[drr]\ar@{-}[dll]&&\\
 *{\circ}\ar@{-}[dd]&&&&*{\circ}\ar@{-}[dd]&&&&*{\circ}\ar@{-}[dd]&&&&*{\circ}\ar@{-}[dd]&&&&*{\circ}\ar@{-}[dd]&&&&*{\circ}\ar@{-}[dd]&&&&*{\circ}\ar@{-}[dd]&&&&*{\circ}\ar@{-}[dd]&&&&*{\circ}\ar@{-}[dd]\\
 \\
 *{\bullet}\ar@{-}[drr]&&&&*{\bullet}\ar@{-}[drr]\ar@{-}[dll]&&&&*{\bullet}\ar@{-}[drr]\ar@{-}[dll]&&&&*{\bullet}\ar@{-}[drr]\ar@{-}[dll]&&&&*{\bullet}\ar@{-}[drr]\ar@{-}[dll]&&&&*{\bullet}\ar@{-}[drr]\ar@{-}[dll]&&&&*{\bullet}\ar@{-}[drr]\ar@{-}[dll]&&&&*{\bullet}\ar@{-}[drr]\ar@{-}[dll]&&&&*{\bullet}\ar@{-}[dll]\\
  &&*{\circ}&&&&*{\circ}&&&&*{\circ}&&&&*{\circ}&&&&*{\circ}&&&&*{\circ}&&&&*{\circ}&&&&*{\circ}&&&&*{}\\
}$$
 \caption{Toroidal map of type $(3,3,3)$}\label{hypermap333}
\end{figure}

If the toroidal hypermap is regular then $G$ is the infinite Coxeter group $[3,3,3]$ factorized by either $(\rho_0\rho_1\rho_2\rho_1)^s$ or $(\rho_0\rho_1\rho_2)^{2s} $  depending on whether it is  $(3,3,3)_{(s,0)}$ or $(3,3,3)_{(s,s)}$, respectively.

The map $\{6,3\}_{(s,s)}$ contains $3$ copies of $\{6,3\}_{(s,0)}$ while $\{6,3\}_{(3s,0)}$ contains $3$ copies of $\{6,3\}_{(s,s)}$ and the same relations hold for the corresponding toroidal hypermaps. Particularly, the group of the $(3,3,3)_{(s,0)}$ is a quotient of the group of $(3,3,3)_{(s,s)}$ and the latter is a quotient of the group of $(3,3,3)_{(3s,0)}$.

For the hypermap $(3,3,3)_{(s,0)}$, consider the translations $u = \rho_0\rho_1\rho_2\rho_1$, $v = u^{\rho_1} = \rho_1\rho_0\rho_1\rho_2$ and $t = u^{-1}v$.

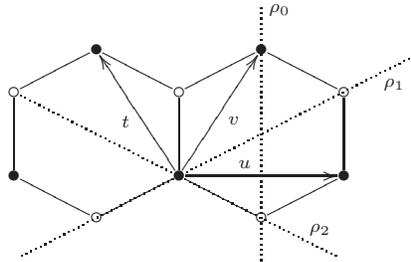
 $$\xymatrix@-1.2pc{
 &&&&&& *{}\ar@{.}[dddddd]^(.01){\rho_0} \\
&&*{\bullet}\ar@{-}[drr]\ar@{-}[dll]&&&&*{\bullet}\ar@{-}\ar@{-}[dll]\ar@{-}[drr]&&&& \ar@{.}[dddddllllllllll]^(.1){\rho_1}\\
*{\circ}\ar@{-}[dd]&&&&*{\circ}\ar@{-}[dd]&&&&*{\circ}\ar@{-}[dd]\\
\\
*{\bullet}\ar@{-}[drr]&&&&*{\bullet}\ar@{->}[rrrr]^(.4){u}\ar@{->}[rruuu]_v\ar@{->}[lluuu]^t\ar@{-}\ar@{-}[dll]\ar@{-}[drr]&&&& *{\bullet}\ar@{-}[dll]\\
&&*{\circ}&&&&*{\circ}\\
&&&&&&*{}&&\ar@{.}[uuuullllllll]_(.1){\rho_2}\\
 }$$

 We have the equalities 
 \begin{equation}\label{relationS0}
  u^{\rho_0} = u^{-1},\ u^{\rho_2} = t^{-1},\ v^{\rho_2} = v^{-1},\ v^{\rho_0} = t\ \text{ and }\ t^{\rho_1} = t^{-1}.
 \end{equation}

For the hypermap $(3,3,3)_{(s,s)}$ consider the translations $g := uv = (\rho_0\rho_1\rho_2)^2$, $h := g^{\rho_0} = $ and $j := gh$.

 $$
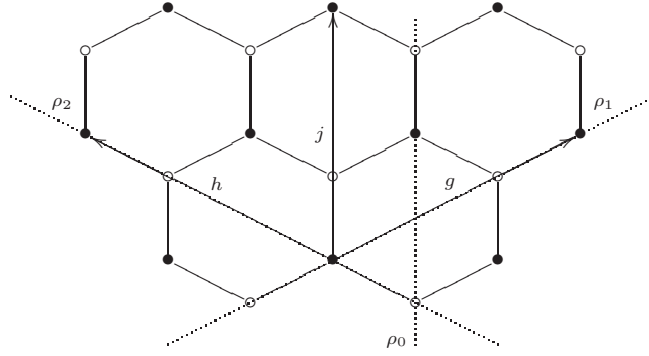
\xymatrix@-1.2pc{
&&&&*{\bullet}\ar@{-}[drr]\ar@{-}[dll]&&&&*{\bullet}\ar@{-}[drr]\ar@{-}[dll]&&&&*{\bullet}\ar@{-}[drr]\ar@{-}[dll]\\
 &&*{\circ}\ar@{-}[dd]&&&&*{\circ}\ar@{-}[dd]&&&&*{\circ}\ar@{-}[dd]&&&&*{\circ}\ar@{-}[dd]\\
 &&&& &&&& &&&& &&&&\\
 &&*{\bullet}\ar@{-}[drr]&&&&*{\bullet}\ar@{-}[drr]\ar@{-}[dll]&&&&*{\bullet}\ar@{-}[drr]\ar@{-}[dll]&&&&*{\bullet}\ar@{-}[dll]\\
 && &&*{\circ}\ar@{-}[dd]&&&&*{\circ}\ar@{-}[dd]&&&&*{\circ}\ar@{-}[dd]\\
  \\
  &&&&*{\bullet}\ar@{-}[drr]&&&&*{\bullet}\ar@{-}[drr]\ar@{-}[dll]\ar@{->}[rrrrrruuu]^g \ar@{->}[lllllluuu]_h \ar@{->}[uuuuuu]^j&&&&*{\bullet}\ar@{-}[dll] \\
   &&&&&&*{\circ}&&&&*{\circ}\\   
   &&&& *{}\ar@{.}[uuuurrrurrrrrrrurr]^(.9){\rho_1} &&&& && *{}\ar@{.}[uuuuuuuu]^(.01){\rho_0} && *{}\ar@{.}[uuuulllulllllllull]_(.9){\rho_2}\\
  }$$

In this case we have the following equalities

\begin{equation}\label{relationSS}
 g^{\rho_1} = g,\ g^{\rho_2} = j^{-1}\ \text{ and }\ h^{\rho_1} = j^{-1}.
\end{equation}

\section{Degrees of maps of type $\{6,3\}$ vs. degrees of toroidal hypermaps}

The degrees of a faithful transitive permutation representation of the group of a regular map of type $\{3,6\}$ (or equivalently $\{6,3\}$) are given in  \cite{2019FP} by the following two theorems.

\begin{theorem}\cite{2019FP}[Theorem 5.6]\label{T3}
Let $s>2$. There exists a CPR graph of a toroidal map $\{3,6\}_{(s,0)}$ with $n$ vertices if and only if $n$ is $s^2$, $2s^2$, $4s^2$, or is equal to $3ds$ with $d$ a divisor of $s$, or $n$ is either $6ab$ or $12ab$ where $a$ and $b$ are positive integers with $s=lcm(a,b)$. \end{theorem}

\begin{theorem}\cite{2019FP}[Theorem 5.7]\label{T4}
Let $s\geq 2$. There exists a CPR graph of a toroidal map $\{3,6\}_{(s,s)}$ with $n$ vertices if and only if $n$ is $3s^2, \,6s^2$,  $12s^2$ or $n$ is $9ds$ with $d$ a divisor of $s$, or is either $18ab$ or $36ab$ where $a$ and $b$ are positive integers with $s=lcm(a,b)$. \end{theorem}

As seen in \cite{2019FP} there is a correspondence between core-free subgroups and faithful transitive actions. Moreover, if $G$ has a faithful transitive permutation representation of degree  $n$ and is a subgroup of index $\alpha$ of $K$, then $K$ has a faithful transitive permutation representation of degree $\alpha n$. 
Similarly to Corollary 3.5 of \cite{2019FP} we have the following.

 \begin{corollary}\label{degreesss}
 If $n$ is a degree of $(3,3,3)_{(s,0)}$ (resp. $(3,3,3)_{(s,s)}$) then $3n$ is a degree of  $(3,3,3)_{(s,s)}$ (resp. $(3,3,3)_{(3s,0)}$).
\end{corollary} 

Additionaly the group of symmetries of  a toroidal hypermap $(3,3,3)_{(s,t)}$ is a subgroup of index 2 of the group of the toroidal map $\{6,3\}_{(s,t)}$ hence we have the following. 

\begin{corollary}\label{degrees63}
 If $n$ is a degree of $(3,3,3)_{(s,0)}$ (resp. $(3,3,3)_{(s,s)}$) then $2n$ is a degree of $\{6,3\}_{(s,0)}$ (resp. $\{6,3\}_{(s,s)}$). 
\end{corollary} 

It must be pointed out that this property works only in one direction, meaning that a degree $n$ of the group of a map $\{6,3\}_{(s,t)}$ does not determine the degrees of $(3,3,3)_{(s,t)}$. 
By Corollary~\ref{degrees63} the set of possible degrees of $(3,3,3)_{(s,0)}$ and $(3,3,3)_{(s,s)}$ is actually a subset of 
$$\left\{\frac{s^2}{2},\, s^2,\, 2s^2, \,\frac{3ds}{2}, \, 3ab,\, 6ab\right\}\mbox{ and } \left\{\frac{3s^2}{2},\, 3s^2,\, 6s^2, \,\frac{9ds}{2}, \, 9ab,\, 18ab\right\}$$
respectively, where $d$ is a divisor of $s$ and $s=lcm(a,b)$. Of course some of these degrees are not even integers, thus we will get proper subsets of these.
 
Now let $G$ be the group of symmetries of a toroidal regular hypermap and suppose that $G$ is represented as a faithful transitive permutation representation group of degree $n$. Let $u$, $v$, $g$ and $h$ be as in Section~\ref{back} and let $T$ be a subgroup of translations of order $s^2$,  either equal to  $\langle u,v\rangle$ or to $\langle g,h\rangle$ depending whether we are dealing with $(3,3,3)_{(s,0)}$ or  $(3,3,3)_{(s,s)}$. In any case $T$ is a normal subgroup of $G$ that is either transitive or intransitive. Using exactly the same arguments as in the proofs of Proposition 3.3 and  Lemma 3.4 of  \cite{2019FP} we get the following results.

\begin{proposition}\label{Tintss}
If $T$ is transitive then $G$ is the group of $(3,3,3)_{(s,0)}$ and $n=s^2$.
\end{proposition}

\begin{lemma}\label{Gimp}
If $n\neq s^2$, then $G$ is embedded into $S_k\wr S_m$  with $n=km$ $(m,\,k>1)$ and 
\begin{enumerate}
\item[(i)] $k=ab$ where $s=lcm(a,b)$ and,
\item[(ii)] $m$ is a divisor of $\frac{|G|}{s^2}$. 
\end{enumerate}
\end{lemma}

In the previous lemma $a$ and $b$ are the orders of the two prescribed generators of the translation group $T$ in a block.
 \section{The degrees of $(3,3,3)_{(s,0)}$}
 
For the hypermap $(3,3,3)_{(s,0)}$ with $s\geq2$, there are faithful transitive permutation representations of degrees 
$s^2$, $2s^2$, $3s^2$ and $6s^2$, as the dihedral groups $\langle \rho_i, \rho_j\rangle$  and cyclic groups $\langle \rho_i\rho_j\rangle$ are core-free, with $i,j\in\{0,1,2\}$. Similarly to Proposition 5.1 (1) of \cite{2019FP}  $\langle u^a, v^b\rangle$ is a core-free subgroup of $G$. Hence $G$ has a faithful transitive permutation representation of degree $n=6ab$ with $s = lcm(a,b)$.

In what follows we give another core-free subgroup of $G$ index $3ds$ with $d$ a divisor of $s$.

\begin{proposition}\label{core}
Let $G$ be the group of $(3,3,3)_{(s,0)}$ ($s\geq 2$). If $d$ is a divisor of $s$ then $H=\langle u^d \rangle\rtimes \langle \rho_0 \rangle$ is core-free and $|G:H|=3ds$.
\end{proposition}
\begin{proof}
Suppose that $x\in H\cap H^{\rho_1}=\langle u^d \rangle\rtimes \langle \rho_0\rangle\cap \langle v^d\rangle\rtimes \langle \rho_0^{\rho_1}\rangle$.
If $x\notin T$ then $\rho_0 \rho_0^{\rho_1}\in T$, a contradiction. Thus  $x\in T$ and therefore as in (1) we conclude that $x=1$.
The order of $H$ is $\frac{2s}{d}$ thus $|G:H|=3ds$.
\end{proof}

In what follows, we prove that the degrees given previously are the only  possible degrees for the group of symmetries of the map $(3,3,3)_{(s,0)}$ with  $s\geq 2$. By Proposition~\ref{Tintss} we may assume that $T$ is intransitive, therefore, by Lemma~\ref{Gimp}, $G$ is embedded into $S_k\wr S_m$  where $n=km$ with $m\in\{2,\,3,\,6\}$ ( $m$ being the number of orbits of $T=\langle u,v\rangle$). Moreover $k=ab$ with $s=lcm(a,b)$. As we found a core-free subgroup of $G$ of index $6ab$, only the cases $m=2$ and $m=3$ need to be considered.

\begin{proposition}\label{m=2}
 If $m=2$, then $k=s^2$.
\end{proposition}
\begin{proof}
If $m=2$ then $T$ has two orbits of size $k=ab$, with $s=lcm(a,b)$, and $G$ has a core-free subgroup $H$ of index $2ab$. 
But then $H$ is also a core-free subgroup of the group of the map $\{6,3\}_{(s,0)}$, of index $4ab$.
Let $G/H=\{Hg_1,\ldots, Hg_n\}$ and let $K$ be the group of symmetries of $\{6,3\}_{(s,0)}$. Then $K/H=\{Hg_1,\ldots, Hg_n\}\cup \{H\rho_0g_1,\ldots, H\rho_0g_n\}$.
As $\{Hg_1,\ldots, Hg_n\}$ and $\{H\rho_0g_1,\ldots, H\rho_0g_n\}$ are in different $T$-orbits, the action of $K$ on $K/H$ gives
a faithful transitive permutation representation for the map $\{6,3\}_{(s,0)}$ for which $T$ has 4 orbits of size $k=ab$.
But then by Proposition~5.5 of \cite{2019FP} $k=s^2$.
\end{proof}

\begin{proposition}\label{m=3}
 If $m=3$, then $k=ds$, for some divisor $d$ of $s$.
\end{proposition}
\begin{proof}
 Consider the action of $G$ on the three orbits, $B_i$ ($i\in\{1,2,3\}$), of the translation subgroup $T$. 
 As $(\rho_i\rho_j)^3=1$ for $i\neq j$ and $u$ and $v$ fix the blocks we find only one possibility for such action, represented by the following graph.
 
 $$
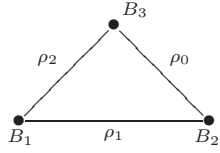
\xymatrix@-1.pc{
 && *{\bullet}\ar@{-}[ddll]_{\rho_2} \ar@{-}[ddrr]^{\rho_0}^(.001){B_3} && \\
 \\
 *{\bullet}\ar@{-}[rrrr]_{\rho_1}_(.99){B_2}_(.01){B_1}&&&& *{\bullet}\\
 }$$
 
 Assume $k=ab \neq s$. Let $u_i$, $v_i$ and $t_i$ denote the action of $u$, $v$ and $t:=u^{-1}v$ on the block $B_i$, respectively. 
 Suppose (without lost of generality) that $|u_1|=a$, $|v_1|=b$ and $|t_1|=lcm(ab)=s$. 
 As $u^{\rho_1} = v$, then $|u_2|=b$ and $|v_2|=a$. Also, we know that $u^{\rho_2} = t^{-1}$, implying that $|u_3|=s$ and $|t_3|=a$.
 Since we know that $u^{\rho_0} = u^{-1}$, we have that $u_2^{\rho_0} = u_3^{-1}$, i.e., $|u_2|=|u_2^{\rho_0}| = |u_3| \Leftrightarrow s = b$. Therefore $lcm(a,b)=b=s$ which implies that $k=ds$ for some divisor $d$ ($=a$) of $s$.
\end{proof}

\begin{theorem}\label{T3}
 Let $s\geq 2$. The toroidal hypermap $(3,3,3)_{(s,0)}$  has a faithful transitive permutation representation of degree $n$ if and only if 
 $$n\in\{s^2,\, 2s^2,\, 3ds, \, 6ab\}$$ 
 where $d$ is a divisor of $s$ and $a$ and $b$ are positive integers such that $s=lcm(a,b)$.
\end{theorem}
\begin{proof}
 This is a consequence of the observations made at the beginning of this section and Proposition~\ref{core},  Lemmas~\ref{Tintss} and \ref{Gimp}, Propositions~\ref{m=2} and \ref{m=3}.
\end{proof}

\begin{corollary}\label{blocks}
Let $m$ be the number of orbits of the group of translation $T=\langle u,v\rangle$.
\begin{enumerate}
\item If $n=s^2$ then $m=1$.
\item If $n=2s^2$ then $m=2$.
\item If $n=3ds$ then $m=3$ and $u^d$ fixes a $T$-orbit point-wisely;
\item If $n=6ab$ then $m=6$ and $u^a$ and $v^b$ fixes a $T$-orbit point-wisely.
\end{enumerate}
\end{corollary}

\section{The degrees of $(3,3,3)_{(s,s)}$}

In this section we determine the degrees of $(3,3,3)_{(s,s)}$ using the degrees of $(3,3,3)_{(s,0)}$ and  $(3,3,3)_{(3s,0)}$, given in Theorem~\ref{T3}.
Let $G$ be the group of $(3,3,3)_{(s,s)}$.

\begin{theorem}\label{T4}
Let $s\geq 2$. A faithful transitive permutation representation of the group of symmetries of $(3,3,3)_{(s,s)}$ has degree $n$ if and only if
 $$n\in\{3s^2, \,6s^2, \,9ds,\,18ab\}$$
  with $s=lcm(a,b)$ and $d$ a divisor of $s$. 
  \end{theorem}
\begin{proof}
 From Theorem~\ref{T3} and Corollary~\ref{degreesss} there are faithful transitive permutation representations with the degrees given in the statement of this theorem. 
 By Theorem~\ref{T3} the possible degrees for $(3,3,3)_{(3s,0)}$  are  
 $$(3s)^2, \,2(3s)^2, \,3\delta (3s),\,6\alpha\beta$$
 with $\delta$ dividing $3s$ and $lcm(\alpha,\beta)=3s$. 
 Moreover Corollary~\ref{blocks} determines the orbits of the normal subgroup $T=\langle u,v\rangle$ of the group $G$ of translations of $(3,3,3)_{(3s,0)}$. 
 
 Consider the embedding of $(3,3,3)_{(s,s)}$ into  $(3,3,3)_{(3s,0)}$.  Dividing the degrees given above by $3$,  we get that the degrees of the group of  $(3,3,3)_{(s,s)}$ belong to the set 
 $$ \{3s^2, \,6s^2, \, 3\delta s,\, 2\alpha\beta\}$$
 with $\delta$ dividing $3s$ and $lcm(\alpha,\beta)=3s$.  
The degrees $3s^2$ and $6s^2$ are in the set given in the statement of the theorem. 
We need only to deal with the case $n\in\{ 3\delta s,\, 2\alpha\beta\}$  with $\delta$ dividing $3s$ and $lcm(\alpha,\beta)=3s$. 

 The hypermap $(3,3,3)_{(3s,0)}$ contains three copies of the hypermap $(3,3,3)_{(s,s)}$. To be more precise the group of  $(3,3,3)_{(s,s)}$ is the group of $(3,3,3)_{(3s,0)}$ factorized by the translation $(uv)^s$ of order $3$. Hence, the points $x$,  $x(uv)^s$ and $x(uv)^{2s}$ of any faithful transitive permutation representation of $(3,3,3)_{(3s,0)}$ are identified under this factorization. Any faithful transitive permutation representation of an action of $(3,3,3)_{(3s,0)}$ on a set $X$ gives a permutation representation, of degree $|X|/3$, of $(3,3,3)_{(s,s)}$ on triples of points of $X$ of the form 
 $$\left\{x, x(uv)^s, x(uv)^{2s}\right\}.$$
 with $x\in X$. But there is no guarantee that this action is faithful.
 
Consider separately the cases: (1) $n= 3\delta s$, with $\delta$ a divisor of $3s$; (2) $n=2\alpha\beta$ with $3s=lcm(\alpha,\beta)$.

(1) If $\delta$ does not divide $s$ then $n=9ds$ with $d$ a divisor of $s$, which is one of the degrees given in the statement of this theorem. 
Suppose that $\delta$ divides $s$. Let $B$ be one of the three blocks of size $3s\delta$ of the faithful transitive permutation representation of degree $3\delta(3s)$ of the group of $(3,3,3)_{(3s,0)}$. We know, by Corollary~\ref{blocks}, that $u^\delta$ fixes every point in $B$, hence also $u^s$ acts like the identity on $B$.
To get the corresponding permutation representation of $(3,3,3)_{(s,s)}$, we use the identification modulo $ (uv)^s$ and one get  an action of the group of  $(3,3,3)_{(s,s)}$ on triples of points  $\left\{x, x(uv)^s, x(uv)^{2s}\right\}$ that belong to the same block. Since $u^s$ fixes every point of $B$, the triple of points of $B$ are of the form
$$\left\{x, x v^s, xv^{2s}\right\}.$$
Hence $v^s$ fixes every triple of points in $B$. Then by conjugation by $\rho_0$, $\rho_1$ and $\rho_2$, we get $u^s$ and $v^s$ fixing every triple of points of the form $\left\{x, x(uv)^s, x(uv)^{2s}\right\}$. 
But the order of $u$ and $v$ in $(3,3,3)_{(s,s)}$ is $3s$ meaning that the action on the triples does not give a faithful permutation representation of $(3,3,3)_{(s,s)}$.

(2) In this case there is, by Corollary~\ref{blocks}, a block $B$ where $u^{\alpha}$ and $v^{\beta}$ act as the identity. Consider that neither $\alpha$ nor $\beta$ divide $s$. Since both are divisors of $3s$, we have $\alpha = 3a$ and $\beta = 3b$, where $lcm(a,b) = s$, giving the degree $18ab$, which is one of the degrees given in the statement of this theorem. Now suppose that either $\alpha$ or $\beta$ is a divisor $s$. Without loss of generality assume that $\alpha$ is a divisor of $s$. Then as in (1) both, $u^s$ and $v^s$, fix every triple thus the action on the triples is not faithful.
\end{proof}

\section{Open Problems}

The study of faithful transitive permutation representations can be extended to other regular polytopes, particularly to 
finite locally spherical regular polytopes,  including the cubic tesselations and to the finite locally toroidal regular polytopes.
 
 \begin{problem}
 Determine the degrees of faithful transitive permutation representations of the groups of spherical and euclidean type.
\end{problem}

\begin{problem}\label{2}
Determine the degrees of faithful transitive permutation representations of the groups of the finite toroidal regular polytopes.
\end{problem}

The problem of the classification locally toroidal regular polytopes dominated the theory of abstract polytopes for a while and it was originally posed by Gr\"{u}nbaum \cite{G77}.
The meritoriously known as Gr\"{u}nbaum's Problem, is not yet totally solved  \cite{ARP}.

\section{Acknowledgements}
This work is supported by The Center for Research and Development
in Mathematics and Applications (CIDMA) through the Portuguese
Foundation for Science and Technology
(FCT - Fundação para a Ciência e a Tecnologia),
references UIDB/04106/2020 and UIDP/04106/2020.

\bibliographystyle{acm}

\end{document}